\documentclass[11pt,a4paper]{amsart}
\usepackage[text={16cm,24.5cm}]{geometry}

\usepackage{amsfonts, amsmath, amssymb, amsthm}
\usepackage[mathscr]{eucal}
\usepackage{hyperref}
\usepackage{graphicx,color}

\usepackage{booktabs}
\usepackage{bbm}
\usepackage[latin1]{inputenc}
\usepackage{psfrag}
\usepackage{overpic}
\usepackage{enumerate}
\usepackage{subfigure}
\usepackage{verbatim}
\usepackage{floatfig}

\newcommand{\NN}{{\mathbb{N}}}

\newcommand{\RR}{{\mathbb{R}}}
\newcommand{\KK}{{\mathbb{K}}}
\newcommand{\cH}{{\mathcal{H}}}
\newcommand{\cS}{{\mathcal{S}}}
\newcommand{\cT}{{\mathcal{T}}}
\newcommand{\cZ}{{\mathcal{Z}}}

\newcommand{\TA}{{\mathbb{T}}}
\newcommand{\trop}{\mathrm{trop}}

\newcommand{\bZero}{\mathbf{0}}
\newcommand{\SetOf}[2]{\{#1\vphantom{#2} \mid \vphantom{#1}#2\}}
\newcommand{\SetOfbig}[2]{\bigl\{#1\vphantom{#2} \mid \vphantom{#1}#2\bigr\}}

\newcommand{\tconv}{\operatorname{tconv}}

\newcommand{\type}{\operatorname{type}}
\newcommand{\Sym}{\operatorname{Sym}}

\DeclareMathOperator{\V}{Vert}
\DeclareMathOperator{\PV}{PV}
\newtheorem{theorem}{Theorem}

\newtheorem{corollary}[theorem]{Corollary}
\newtheorem{lemma}[theorem]{Lemma}

\newtheorem{example}[theorem]{Example}

\begingroup
\theoremstyle{plain}

\endgroup

\begin{document}

\title[Coarse Types Of Tropical Matroid Polytopes]{Coarse Types Of Tropical Matroid Polytopes}

\author[Kulas]{Katja Kulas}
\address{Fachbereich Mathematik, TU Darmstadt, 64293 Darmstadt, Germany}
\email{kulas@mathematik.tu-darmstadt.de}

\begin{abstract}
  Describing the combinatorial structure of the tropical complex
  $\mathcal{C}$ of a tropical matroid polytope, we obtain a formula for the coarse
  types of the maximal cells of $\mathcal{C}$. Due to the connection between
  tropical complexes and resolutions of monomial ideals, this yields the generators for
  the corresponding coarse type ideal introduced in \cite{DJS09}. Furthermore, a
  complete description of the minimal tropical halfspaces of the uniform tropical 
  matroid polytopes, i.e. the tropical hypersimplices, is given.
\end{abstract}

\maketitle

\section{Introduction}
Tropical matroid polytopes have been introduced in~\cite{DSS2005} as the
tropical convex hull of the cocircuits, or dually, of the bases of a matroid. 
The arrangement of finitely many points $V$ in the
tropical torus $\TA^d$ has a natural decomposition $\mathcal{C}_V$ of $\TA^d$ into (ordinary)
polytopes, the tropical complex, equipped with a (fine) type $T$, which encodes the relative position to the generating
points. The \emph{coarse types} only count the cardinalities of $T$. In~\cite{DS04}, Develin and Sturmfels showed that
the bounded cells of $\mathcal{C}_V$ yield the tropical convex hull
of $V$, which is dual to the regular subdivision $\Sigma$ of a product of two
simplices (or equivalently---due to the Cayley Trick---to the regular mixed subdivisions of a dilated
simplex).  
The authors of ~\cite{BlockYu} and~\cite{DJS09} use the
connection of the cellular structure of $\mathcal{C}_V$ or rather of $\Sigma$ to
minimal cellular resolutions of certain monomial ideals to provide an
algorithm for determining the facial structure of the bounded subcomplex
 of $\mathcal{C}_V$. A main result of~\cite{DJS09} says that the labeled complex
$\mathcal{C}_V$ supports a minimal cellular
resolution of the ideal $I$ generated by monomials
corresponding to the set of all (coarse) types.

The main theme of this paper is the study of the tropical complex of tropical
convex polytopes associated with matroids arising from graphs---the \emph{tropical
matroid polytopes}. 
Recall that a {\it matroid} $M$ is a finite collection $\mathcal{F}$ of subsets
of $[n]={1,2,\ldots,n}$, called {\it independent sets}, such that three
properties are satisfied: (i) $\emptyset\in\mathcal{F}$, (ii) if
$X\in\mathcal{F}$ and $Y\subseteq X$ then $Y\in \mathcal{F}$, (iii) if
$U,V\in\mathcal{F}$ and $\lvert U \rvert=\lvert V \rvert+1$ there exists
$x\in U\setminus V$ such that $V\cup x\in\mathcal{F}$. The last one is also
called the \emph{exchange property}. The maximal independent sets are the
\emph{bases} of $M$. A matroid can also be defined by specifying its \emph{non-bases},
i.e. the subsets of $E$ with cardinality $k$ that are not bases. For more
details on matroids see the survey of Oxley
~\cite{Oxley2003} and the books of White(\cite{White1986},~\cite{White1987},~\cite{White1992}). 
An important class of matroids are the graphic or cycle matroids proven to be
regular, that is, they are representable over every field. 
A \emph{graphic matroid} is associated with a simple undirected graph
$G$ by letting $E$ be the set of edges of $G$ and taking as the bases the edges of
the spanning forests. 
Matroid polytopes were first studied in connection with optimization and
linear programming, introduced by Jack Edmonds~\cite{Edmonds03}. 
A nice polytopal characterization for a matroid polytope was given by Gelfand
et~al.~\cite{GGMS1987} stating that each of its edges is a parallel translate of
$e_i-e_j$ for some $i$ and $j$.  

In the case of tropical matroid polytopes the coarse
types display the number $b_{I,J}$ of bases $B$ of the associated matroid with subsets $I,J$,
where all elements of $I$ but none of $J$ are contained in $B$.  
\filbreak
\begin{theorem}
  Let $\mathcal{C}$ be the tropical complex of a tropical matroid polytope with
  $d+1$ elements and rank $k$. The set of all coarse types of the maximal cells arising in $\mathcal{C}$ is
given by the tuples $(t_1,\ldots,t_{d+1})$ with
  \begin{equation*} t_j \ = \
    \begin{cases}
      b_{\{i_1\},\emptyset}+b_{\,\emptyset,\{i_1,i_2,\ldots,i_{d'+1}\}}  & \text{if $j=i_1$} \, ,\\
      b_{\{i_l\},\{i_1,\ldots,i_{l-1}\}} & \text{if $j=i_l\in \{i_2,\ldots i_{d'+1}\}$} \, ,\\
      0 & \text{otherwise} \, .
    \end{cases}
  \end{equation*} 
  where $d'\in [d-k+1]$ and $\{i_1,i_2,\ldots,i_{d'+1}\}$ is a sequence of elements such that
$[d+1]\setminus\{i_1,i_2,\ldots,i_{d'}\}$ contains a basis of the associated
matroid.
\end{theorem}
Subsequently, we relate our combinatorial result to commutative algebra. For the
coarse type $\mathbf{t}(p)$ of $p$  and
$x^{\mathbf{t}(p)}={x_1}^{{\mathbf{t}(p)}_1}{x_2}^{{\mathbf{t}(p)}_2}\cdots{x_{d+1}}^{{\mathbf{t}(p)}_{d+1}}$ 
the monomial ideal \[I=\langle x^{\mathbf{t}(p)}\colon
p\in\TA^d\rangle\subset\KK[x_1,\ldots,x_{d+1}]\]
is called the \emph{coarse type ideal}. 
In~\cite{DJS09}, Corollary 3.5, it was shown that $I$ is generated by the
monomials, which are assigned to the coarse types of the inclusion-maximal cells of
the tropical complex. As a direct consequence of Theorem 3.6 in~\cite{DJS09}, we obtain the generators of $I$.
\begin{corollary}The coarse type ideal $I$ for the tropical complex of a
tropical matroid polytope with $d+1$ elements and rank $k$
is equal to \[\langle x_{i_1}^{t_{i_1}}x_{i_2}^{t_{i_2}}\cdots x_{i_{d'+1}}^{t_{i_{d'+1}}}\colon
[d+1]\setminus\{i_1,\ldots,i_{d'}\} \text{ contains a basis }\rangle\] where $(t_{i_1},t_{i_2},\ldots,t_{i_{d'+1}})=\big(b_{\{i_1\},\emptyset}+b_{\,\emptyset,\{i_1,i_2,\ldots,i_{d'+1}\}},b_{\{i_2\},\{i_1\}},\ldots,b_{\{i_{d'+1}\},\{i_1,\ldots,i_{d'}\}}\big)$.
\end{corollary}  

Furthermore, we apply these results to the special case of uniform
matroids, introduced and studied in~\cite{Joswig05}. We close this work by
stating the minimal tropical halfspaces containing a uniform tropical matroid polytope
by using the characterization of Proposition 1 in~\cite{GaubertKatz09}.

\section{Basics of tropical convexity}
We start with collecting basic facts about tropical convexity
and fixing the notation. Defining \emph{tropical addition} by $x\oplus y:=\min(x,y)$ and
\emph{tropical multiplication} by $x\odot y:=x+y$ yields the
\emph{tropical semi-ring} $(\RR,\oplus,\odot)$.  Component-wise
tropical addition and \emph{tropical scalar multiplication}
\begin{equation*}
  \lambda \odot (\xi_0,\dots,\xi_d)  :=  (\lambda \odot
  \xi_1,\dots,\lambda \odot \xi_d)  = (\lambda+\xi_0,\dots,\lambda+\xi_d)
\end{equation*}
equips $\RR^{d+1}$ with a semi-module structure.  For
$x,y\in\RR^{d+1}$ the set
\begin{equation*} [x,y]_\trop := \SetOf{(\lambda \odot x) \oplus (\mu
    \odot y)}{\lambda,\mu \in \RR}
\end{equation*}
defines the \emph{tropical line segment} between $x$ and $y$.  A subset of
$\RR^{d+1}$ is \emph{tropically convex} if it contains the tropical
line segment between any two of its points.  A direct computation
shows that if $S\subset\RR^{d+1}$ is tropically convex then $S$ is
closed under tropical scalar multiplication.  This leads to the
definition of the \emph{tropical torus} as the quotient
semi-module
\begin{equation*}
  \TA^d  :=  \RR^{d+1} / (\RR\odot(1,\dots,1))   .
\end{equation*}

Note that $\TA^d$ was called ``tropical projective space'' in
\cite{DS04}, \cite{Joswig05}, \cite{DevelinYu06}, and
\cite{JoswigSturmfelsYu07}.  Tropical convexity gives rise to the hull
operator $\tconv$.  A \emph{tropical polytope} is the tropical convex
hull of finitely many points in $\TA^d$.

Like an ordinary polytope each tropical polytope $P$ has a unique set
of generators which is minimal with respect to inclusion; these are
the \emph{tropical vertices} of $P$.

There are several natural ways to choose a representative coordinate
vector for a point in $\TA^d$.  For instance, in the coset
$x+(\RR\odot(1,\dots,1))$ there is a unique vector $c(x)\in\RR^{d+1}$
with non-negative coordinates such that at least one of them is zero;
we refer to $c(x)$ as the \emph{canonical coordinates} of $x\in\TA^d$.
Moreover, in the same coset there is also a unique vector
$(\xi_0,\dots,\xi_d)$ such that $\xi_0=0$.  Hence, the map
\begin{equation*}\label{eq:c_0}
  c_0   :   \TA^d \to \RR^d   ,  (\xi_1,\dots,\xi_{d+1}) \mapsto (\xi_2-\xi_1,\dots,\xi_{d+1}-\xi_1)
\end{equation*}
is a bijection.  Often we will identify $\TA^d$ with $\RR^d$ via this
map.

The \emph{tropical hyperplane} $\cH_a$ defined by the \emph{tropical
  linear form} $a=(\alpha_1,\dots,\alpha_{d+1})\in\RR^{d+1}$ is the set of
points $(\xi_1,\dots,\xi_{d+1})\in\TA^d$ such that the minimum
\begin{equation*}
  (\alpha_1 \odot \xi_1) \oplus \dots \oplus (\alpha_{d+1} \odot \xi_{d+1})
\end{equation*}
is attained at least twice.  For $d=3$ the tropical hyperplane is shown in
Figure~\ref{fig:3hypersimplicesb}. The complement of a tropical hyperplane
in $\TA^d$ has exactly $d+1$ connected components, each of which is an
\emph{open sector}. A \emph{closed sector} is the topological closure
of an open sector.  The set
\begin{equation*}
  S_k  :=  \SetOfbig{(\xi_1,\dots,\xi_{d+1})\in\TA^d}{\xi_k=0 \text{ and }
    \xi_i>0 \text{ for } i\ne k}   ,
\end{equation*}
for $1\le k \le d+1$, is the \emph{$k$-th open sector} of the tropical
hyperplane $\cZ$ in $\TA^d$ defined by the zero tropical linear form.
Its closure is
\begin{equation*}
  \bar S_k  :=  \SetOfbig{(\xi_1,\dots,\xi_{d+1})\in\TA^d}{\xi_k=0 \text{ and
    } \xi_i\ge 0 \text{ for } i\ne k}   .
\end{equation*}
We also use the notation $\bar S_I:=\bigcup\SetOf{\bar S_i}{i\in I}$
for any set $I\subset[d+1]:=\{1,\dots,d+1\}$.

If $a=(\alpha_1,\dots,\alpha_{d+1})$ is an arbitrary tropical linear form
then the translates $-a+S_k$ for $1\le k\le d+1$ are the open sectors of
the tropical hyperplane $\cH_a$.  The point $-a$ is the unique point
contained in all closed sectors of $\cH_a$, and it is called the
\emph{apex} of $\cH_a$.  For each $I\subset[d+1]$ with $1\le
\# I \le d$ the set $-a+\bar S_I$ is the \emph{closed tropical
  halfspace} of $\cH_a$ of type $I$. A tropical halfspace $H(-a,I)$ can
also be written in the form
\begin{eqnarray*}
  H(-a,I)&=&\{x\in\TA^d\mid\text{ the minimum of }\displaystyle\bigoplus_{i=1}^{d+1} \alpha_i\odot
  \xi_i\text{ is attained}\\
  &&\,\,\text{ at a coordinate }i\in I\}\\&=&\{x\in\TA^d\mid \displaystyle\bigoplus_{i\in
    I} (\alpha_i\odot \xi_i)\leq\bigoplus_{j\in
    J} (\alpha_j\odot \xi_j)\}\end{eqnarray*}where $I$ and $J$ are disjoint subsets of
$[d+1]$ and $I\cup J= [d+1]$. The tropical polytopes in
$\TA^d$ are exactly the bounded intersections of finitely many closed
tropical halfspaces; see \cite{GaubertKatz09} and \cite{Joswig05}.

We concentrate on the combinatorial structure of tropical polytopes. Let
$V:=(v_1,\dots,v_n)$ be a sequence of points in $\TA^d$. 
The \emph{(fine) type} of $x\in\TA^d$ with respect to $V$ is the ordered
$(d+1)$-tuple $\type_V(x):=(T_1,\dots,T_{d+1})$ where
\begin{equation*}
  T_k  :=  \SetOf{i\in\{1,\dots,n\}}{v_i\in x+\bar S_k}   .
\end{equation*}
For a given type $\cT$ with respect to $V$ the set
\begin{equation*}
  X^{\circ}_V(\cT)  :=  \SetOfbig{x\in\TA^d}{\type_V(x)=\cT}
\end{equation*}
is a relatively open subset of $\TA^d$ and is called the \emph{cell} of
type $\cT$ with respect to $V$. The set $X^{\circ}_V(\cT)$ as well as its
topological closure are tropically and ordinary convex; in \cite{JK08}, these
were called \emph{polytropes}. With respect
to inclusion the types with respect to $V$ form a partially ordered
set. The intersection of two cells $X_V(\cS)$ and $X_V(\cT)$ with
type $\cS$ and $\cT$ is equal to the polyhedron $X_V(\cS\cup\cT)$. 
The collection of all (closed) cells induces a polyhedral subdivision
$\mathcal{C}_V$ of $\TA^d$. A $\min$-tropical polytope $P=\tconv(V)$ is the
union of cells in the bounded subcomplex $\mathcal{B}_V$ of $\mathcal{C}_V$
induced by the arrangement $\mathcal{A}_V$ of $\max$-tropical hyperplanes with apices $v\in
V$. A cell of $\mathcal{C}_V$ is unbounded if and only if one of its type components is 
the empty set. The type of $x$ equals the union of the types of the (maximal) 
cells that contain $x$ in their closure. The dimension of a cell $X_T$ can be
calculated as the
number of the connected components of the undirected graph 
$G=\big(\{1,2,\ldots,d+1\},\,\{(j,k)\mid T_j\cap T_k\neq\emptyset\}\big)$
minus one. The zero-dimensional cells 
are called pseudovertices of $P$.

Replacing the (fine) type entries $T_k\subseteq [n]$ for $k\in [d+1]$ of a point $p\in\TA^d$ by
their cardinalities $t_k:=\left|T_k\right|$ we get the \emph{coarse type} 	
$t_V(p)=(t_1,\ldots,t_{d+1})\in\NN^{d+1}$ of $p$. 
A coarse type entry $t_k$ displays how many generating points lie in the $k$-th closed sector
$p+\overline{S_k}$.  
In~\cite{DJS09}, the authors associate the tropical complex of a tropical
polytope with a monomial ideal, the coarse type ideal \[I:=\langle
{x_1}^{t_1}{x_2}^{t_2}\cdots{x_{d+1}}^{t_{d+1}}\colon
p\in\TA^d\rangle\subset\KK[x_1,\ldots,x_{d+1}].\] By Corollary 3.5 of~\cite{DJS09}, $I$ is generated by the
monomials assigned to the coarse types of the inclusion-maximal cells of
the tropical complex. The tropical complex
$\mathcal{C}_V$ gives rise to minimal cellular resolutions of $I$.
\begin{theorem}[ \cite{DJS09}, Theorem 3.6 ]\label{thm:DJS2009} The labeled complex
$\mathcal{C}_V$ supports a
minimal cellular resolution of the ideal $I$ generated by monomials
corresponding to the set of all (coarse) types.
\end{theorem}
Considering cellular resolutions of monomial ideals, introduced in~\cite{BPS1998} and~\cite{BS1998}, is a natural
technique to construct resolutions of monomial ideals using labeled cellular
complexes and provide an important interface
between topological constructions, combinatorics and algebraic ideas. The authors
of~\cite{BlockYu} and \cite{DJS09} use this to give an algorithm for determining the facial structure of a
tropical complex. More precisely, they associate a squarefree monomial ideal $I$ with a tropical
polytope and calculate a minimal
cellular resolution of $I$, where the $i$-th syzygies of $I$ are encoded by the
$i$-dimensional faces of a polyhedral complex.

A tropical halfspace is called \emph{minimal} for a tropical polytope $P$ if 
it is minimal with respect to inclusion among all tropical halfspaces containing
$P$. Consider a tropical halfspace $H(a,I)\subset \TA^d$ with $I\subset[d+1]$ and apex 
$a\in \TA^d$, and a tropical polytope $P=\tconv\{v_1,\ldots,v_n\}\subseteq \TA^d$.
To show that $H(a,I)$ is minimal for $P$, it suffices to prove, by Proposition 1
of~\cite{GaubertKatz09}, that the following three
  criteria hold for the type $(T_1,T_2,\ldots,T_{d+1})=\type_V(a)$ of the apex $a$: 
\begin{itemize}\item[(i)] $\displaystyle\bigcup_{i\in I}T_i=[n]$, \item[(ii)] for each
  $j\in I^C$ there exists an $i\in I$ such that $T_i\cap
  T_j\neq\emptyset$, \item[(iii)] for each $i\in I$ there exists $j\in I^C$ such that
  $\displaystyle T_i\cap T_j\not\subset\bigcup_{k\in I\setminus\{i\}}T_k$.\end{itemize}
  Here, we denote the complement of a set $I\subseteq [d+1]$ as
  $I^C=[d+1]\setminus I$.

\noindent
Obvious minimal tropical halfspaces of a tropical polytope
$P=\tconv(V)\subseteq \TA^{d}$ are its cornered halfspaces, see~\cite{Joswig08}.  
The {\it $k$-th corner} of $P$ is defined as 
\[c_k(V):= (-v_{1,k}) \odot v_1\oplus(-v_{2,k})\odot v_2
\oplus\cdots\oplus(-v_{n,k})\odot v_n.\] 
The tropical halfspace $H_k:=c_k(V)+\overline{S_k}$ is called the 
{\it $k$-th cornered tropical halfspace} of $P$ and the intersection 
of all $d+1$ cornered halfspaces is the {\it cornered hull} of $P$.

\section{Tropical Matroid Polytopes} 
The tropical matroid polytope of a matroid $\mathcal{M}$ 
is defined in~\cite{DSS2005} as the tropical convex hull of the negative incidence vectors 
of the bases of $\mathcal{M}$. In this paper, we restrict ourselves to matroids arising
from graphs. 

The \emph{graphic matroid} of a simple undirected graph $G=(V,E)$ is
$\mathcal{M}(G)=(E,\mathcal{I}=\{F\subseteq E\colon F \text{ is acyclic}\})$. 
While the forests of $G$ form the system of independent sets of
$\mathcal{M}(G)$ its bases are the spanning forests. We will assume that $G$ is
connected, so the bases of $\mathcal{M}(G)$ are the spanning trees of
$G$. Furthermore, we exclude bridges, i.e. edges whose deletion increases the
number of connected components of $G$, leading to elements that are contained 
in every basis. 
Let $d+1$ be the number of elements and $n$ be the number of bases of 
$\mathcal{M}:=\mathcal{M}(G)$ and $\mathcal{B}:=\{B_1,\ldots,B_n\}$ its bases. 
It follows from the exchange property of matroids that all bases of 
$\mathcal{M}$ have the same number of elements, which is called the
\emph{rank} of $\mathcal{M}$. Consider the $0/1$-matrix 
$M\in\RR^{(d+1)\times n}$ with rows indexed by the elements of the 
ground set $E$ and columns indexed by the bases of $\mathcal{M}$ 
which has a $0$ in entry $(i,j)$ if the $i$-th element is in the $j$-th
basis. The \emph{tropical matroid polytope} $P$ of $\mathcal{M}$ is the
tropical convex hull of the columns of $M$. Let
\begin{equation}\label{eq:generators}V=\left\{-e_B:=\sum_{i\in
      B}-e_i\big|B\in \mathcal{B}\right\}\end{equation} be the set of generators of
$P$. It turns out that these are just the tropical vertices of $P$, see Lemma~\ref{lem:pseudovertices}.
If the underlying matroid has rank $k$, then the canonical coordinate 
vectors of $V$ have exactly $k$ zeros and $d+1-k$ ones and
will be denoted as $v_{B_i}$ or for short $v_i$ if the corresponding basis is
$B_i\in\mathcal{B}$. Note that with $\oplus$ as $\max$ instead of $\min$ the generators 
of a tropical matroid polytope are the positive incidence vectors of
the bases of the corresponding matroid. Throughout this paper we write $\mathcal{P}_{k,d}$ for the set of all tropical matroid polytopes arising from a
graphic matroid with $d+1$ elements and rank $k$.  

\begin{example}
  The tropical hypersimplex
  $\Delta_k^d$ in $\TA^d$ studied in \cite{Joswig05} is a 
  tropical matroid polytope of a uniform matroid of rank $k$ with
  $d+1$ elements and $\binom{d+1}{k}$ bases. It is defined as the 
  tropical convex hull of all points $-e_I:=\displaystyle\sum_{i\in I}-e_i$ 
  where $e_i$ is the $i$-th unit vector of $\RR^{d+1}$ and $I$ is a 
  $k$-element subset of $[d+1]$. The tropical vertices of $\Delta_k^d$ are
  \[\V(\Delta_k^d)=\left\{-e_I\big|I\in\binom{[d+1]}{k}\right\}\, \text{ for all }k>0.\] 
  In~\cite{Joswig05}, it is shown that
  $\Delta_{k+1}^d\subsetneq\Delta_k^d$ implying that the first tropical
  hypersimplex contains all other tropical hypersimplices in $\TA^d$. 
  The first tropical hypersimplex $\Delta^d=\Delta_1^d$ in $\TA^d$ is the 
  $d$-dimensional tropical standard simplex which is also a polytrope. 
  Clearly, we have for a tropical matroid polytope $P\in\mathcal{P}_{k,d}$ the
  chain $P\subseteq\Delta_k^d\subsetneq\cdots\subsetneq\Delta_1^d=\Delta^d$.
  For $d=3$ the three tropical hypersimplices are shown in Figure~\ref{fig:3hypersimplices}.
  \begin{figure}[htb]
    \centering
    \subfigure[\label{fig:3hypersimplicesa}$k=3$]{\includegraphics[scale=0.25]{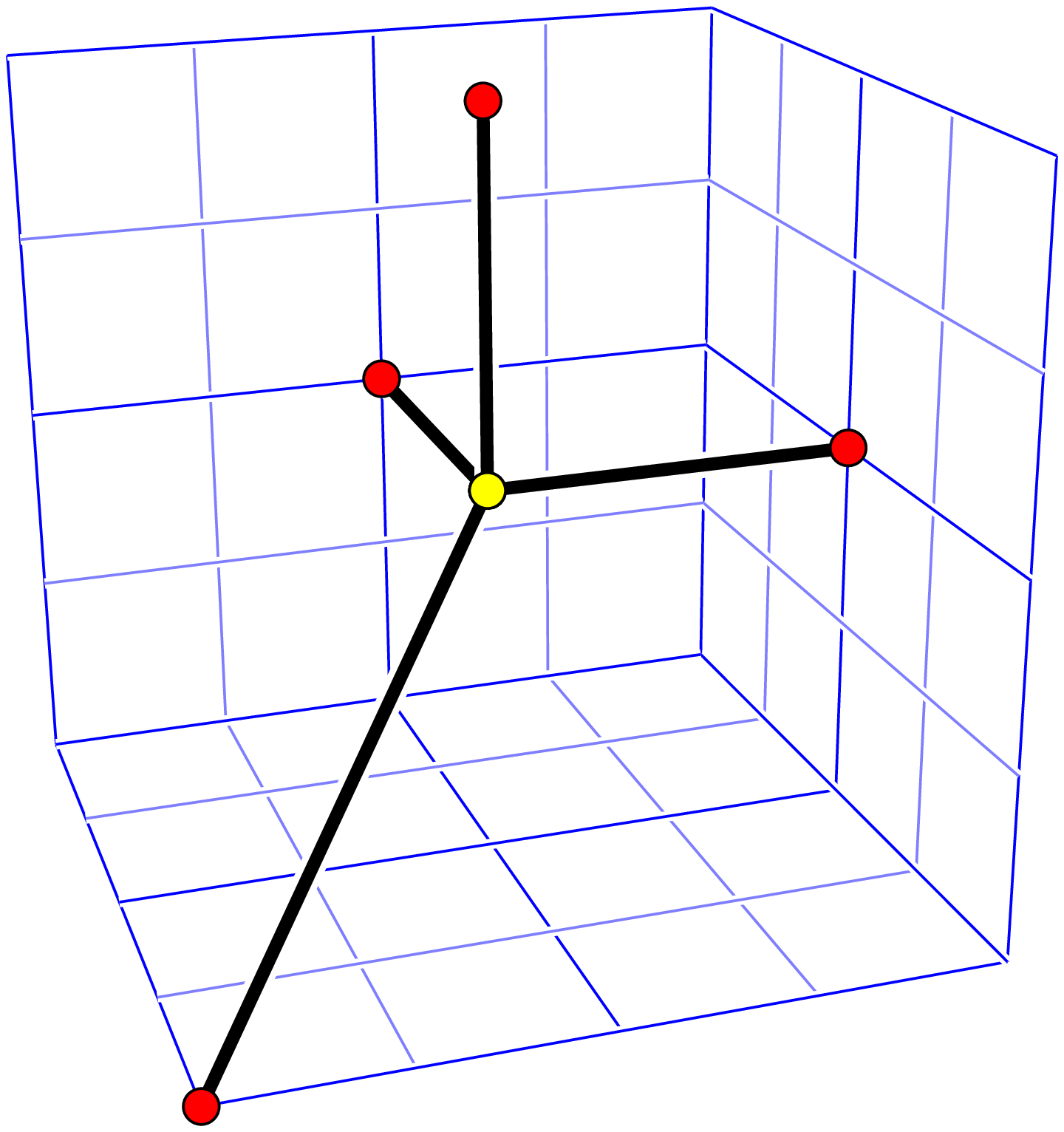}}\hfill
    \subfigure[\label{fig:3hypersimplicesb}$k=2$]{\includegraphics[scale=0.25]{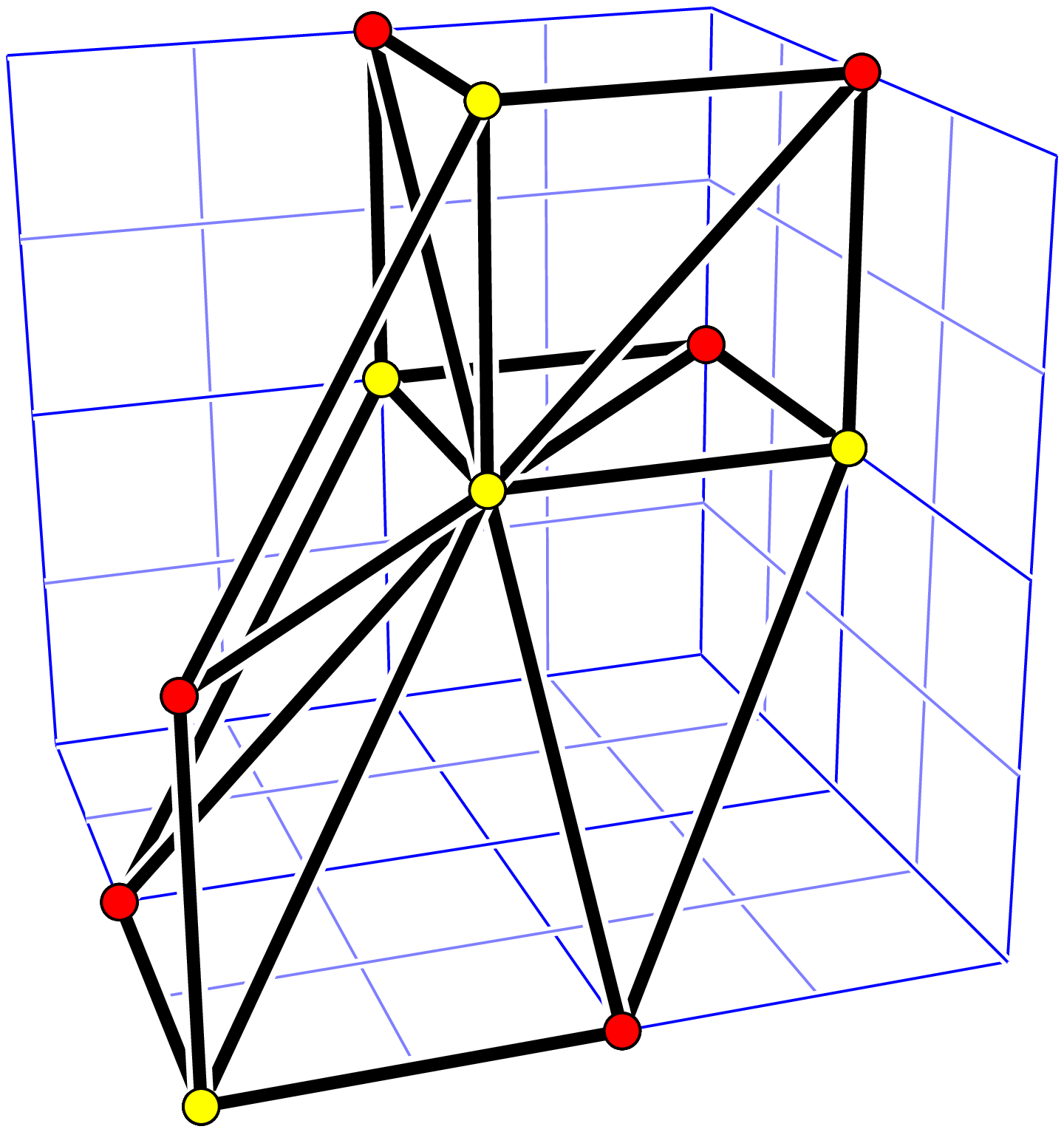}}\hfill
    \subfigure[\label{fig:3hypersimplicesc}$k=1$]{\includegraphics[scale=0.25]{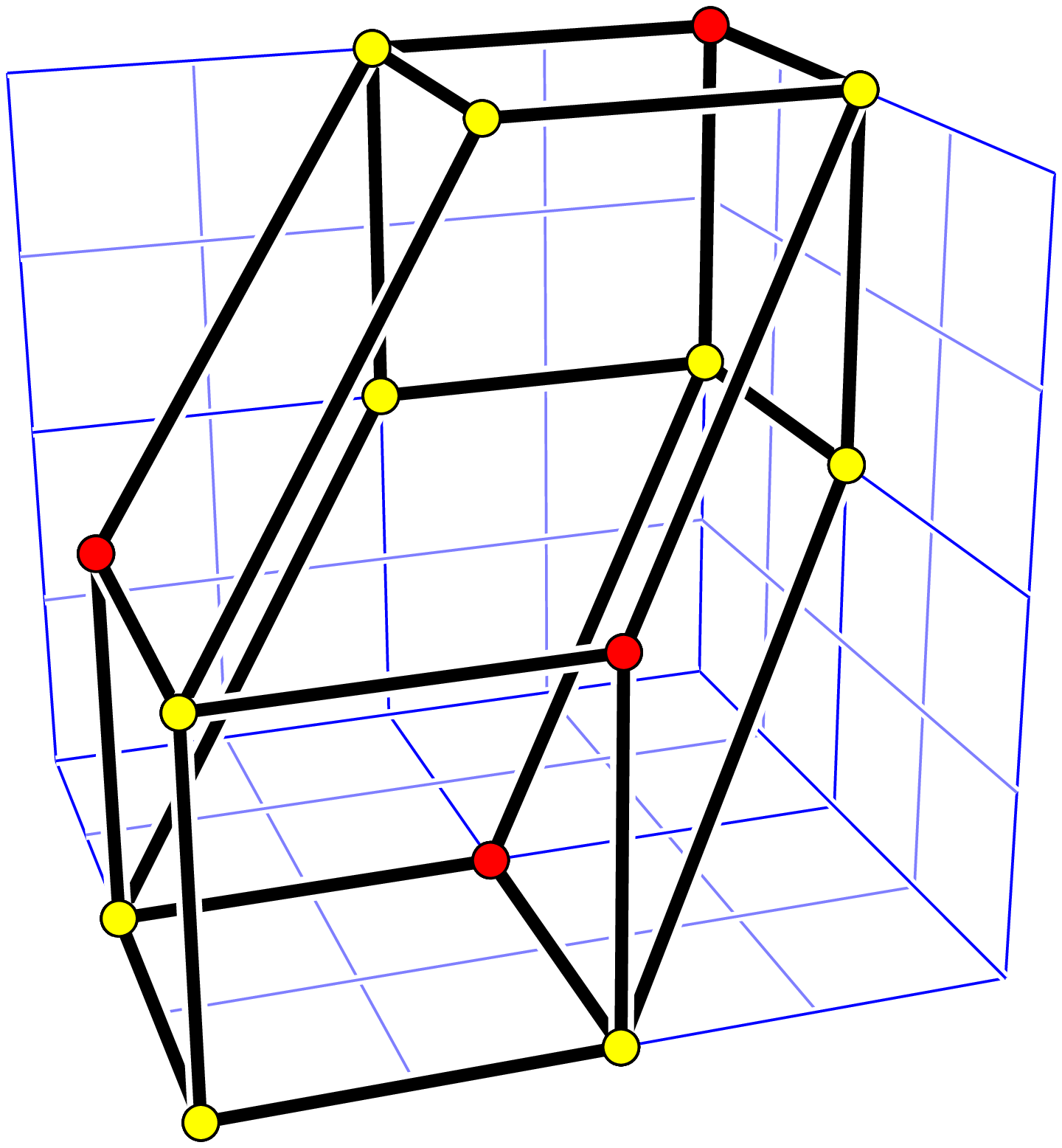}}
    \caption{\label{fig:3hypersimplices}The three $3$-dimensional tropical
hypersimplices with $\Delta_3^3\subset\Delta_2^3\subset\Delta_1^3$.}
  \end{figure}
\end{example} 

The origin $\bZero\in\TA^{d}$ and its fine type are crucial for the 
calculation of the fine and the
coarse types of the maximal cells in the cell complex of $P$.

\begin{lemma}
  A tropical matroid polytope $P\in\mathcal{P}_{k,d}$ with generators $V$ contains the origin
  $\bZero\in\TA^{d}$. Its type is
  $\type_{V}(\bZero)=(T^{(0)}_1,T^{(0)}_2,\ldots,T^{(0)}_{d+1})$ with $T^{(0)}_i=\{j\mid i\in B_j\}$.
\end{lemma}
\begin{proof}
  By Proposition 3 of \cite{DS04} about the shape of a tropical
  line segment, the only
  pseudovertex of the tropical line segment between two distinct $0$-$1$-vectors
  $u$ and $v$ in $\TA^{d}$ is the point $w$ with $w_{l}=0$ if $u_l=0$ or $v_l=0$ and
  $w_{l}=1$ otherwise. Since every element of $E$ is contained in any basis of
  $\mathcal{M}(G)$ (apply any spanning-tree-greedy-algorithm for the connected
  components of $G$ starting from this element) and by using the previous argument,
  the origin must be contained in $P$.  

  An index $j$ is contained in the $i$-th type coordinate $T^{(0)}_i$ if $v_{j,i}=\min\{v_{j,1},v_{j,2},\ldots,v_{j,{d+1}}\}$, which is
  satisfied by all indices $i\in B_j$.
\end{proof}

The $i$-th type entry $T_i^{(0)}$ 
of $\bZero$ contains all bases of $\mathcal{M}$ with element $i$, and $|T_i^{(0)}|$ is
the number of bases of $\mathcal{M}$ containing $i$.

Now it is time to introduce our running example.	

\initfloatingfigs

\begin{example}\label{ex:RunEx1}
  The graphical matroid given by the following graph $G$ has $d+1=5$ elements
  (edges with bold indices), rank $k=3$, $n=8$ bases
  $B_1=\{\mathbf{1},\mathbf{2},\mathbf{4}\},\,B_2=\{\mathbf{1},\mathbf{2},\mathbf{5}\},\,B_3=\{\mathbf{1},\mathbf{3},\mathbf{4}\},\,B_4=\{\mathbf{1},\mathbf{3},\mathbf{5}\},\,B_5=\{\mathbf{1},\mathbf{4},\mathbf{5}\},\,B_6=\{\mathbf{2},\mathbf{3},\mathbf{4}\},\,B_7=\{\mathbf{2},\mathbf{3},\mathbf{5}\},\,B_8=\{\mathbf{3},\mathbf{4},\mathbf{5}\}$
  and the non-bases $\{\mathbf{1},\mathbf{2},\mathbf{3}\},\,\{\mathbf{2},\mathbf{4},\mathbf{5}\}$.
\vspace*{0.1cm}\\
  \begin{minipage}{0.2\textwidth}
    \centering
    \psfrag{0}{$\mathbf{1}$}
    \psfrag{G}{$G:$}
    \psfrag{1}{$\mathbf{2}$}
    \psfrag{2}{$\mathbf{3}$}
    \psfrag{3}{$\mathbf{4}$}
    \psfrag{4}{$\mathbf{5}$}
    \psfrag{T0}{\!$\mathit{(12345)}$}
    \psfrag{T1}{$\mathit{(1267)}$}
    \psfrag{T2}{$\mathit{(34678)}$}
    \psfrag{T3}{$\mathit{(13568)}$}
    \psfrag{T4}{$\mathit{(24578)}$}
    \includegraphics[scale=0.9]{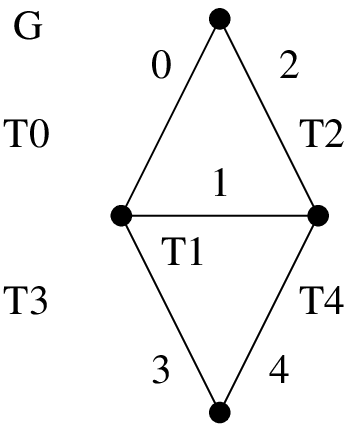}
  \end{minipage} \hfill
  \begin{minipage}{0.7\textwidth}
    Let $P$ be the corresponding tropical matroid polytope with its generators
    {\begin{eqnarray*}V&=&\{v_{B_1},\ldots,v_{B_8}\}\\ &=&\left\{\scriptsize
\begin{pmatrix}0\\0\\1\\0\\1\end{pmatrix},
\begin{pmatrix}0\\0\\1\\1\\0\end{pmatrix},
\begin{pmatrix}0\\1\\0\\0\\1\end{pmatrix},
\begin{pmatrix}0\\1\\0\\1\\0\end{pmatrix},
\begin{pmatrix}0\\1\\1\\0\\0\end{pmatrix},
\begin{pmatrix}1\\0\\0\\0\\1\end{pmatrix},
\begin{pmatrix}1\\0\\0\\1\\0\end{pmatrix},
\begin{pmatrix}1\\1\\0\\0\\0\end{pmatrix}
\right\}.\end{eqnarray*}} The type of the origin $\bZero$ of $P$
    is $(12345,1267,34678,13568,24578)$ where the $i$-th type entry contains all bases
    using the edge $i$ (italic edge attributes). 
  \end{minipage} 
\end{example}
In the next lemma we will show that the tropical standard simplex $\Delta^d$ is 
the cornered hull of all tropical matroid polytopes in $\mathcal{P}_{k,d}$. 
\begin{lemma}\label{lem:chull of tmp}
  The cornered hull of a tropical matroid polytope $P\in\mathcal{P}_{k,d}$ with
  generators $V$ is the  
  $d$-dimensional tropical standard simplex $\Delta^d$. The $i$-th corner 
  of $P$ is the vector $e_i$. The type of $e_i$ with respect to
  $V$ is $\type_{V}(e_i)=(T_1,\ldots,T_{d+1})$ with  
  \begin{equation*} T_j \ = \
    \begin{cases}
      [d+1]  & \text{if $j=i$} \, ,\\
      \{l\mid j\in B_l \text{ and }i\notin B_l \}& \text{otherwise} \, .
    \end{cases}
  \end{equation*}
\end{lemma}
\begin{proof}
  For $B\in\mathcal{B}$ the $i$-th (canonical) coordinate of $v_B$ is \begin{equation*} v_{B,i} \ = \
    \begin{cases}
      0  & \text{if $i\in B$} \, ,\\
      1& \text{otherwise} \, .
    \end{cases}
  \end{equation*}
  The $j$-th coordinate of the $i$-th corner $c_i(V)$ of $P$ is 
  \begin{equation*}c_i(V)_j=\displaystyle\min_{J\in \mathcal{B}}(v_{J,j}-v_{J,i})\ = \
    \begin{cases}
      0  & \text{if $i=j$} \, ,\\
      -1& \text{otherwise} \, .
    \end{cases}
  \end{equation*}
  In canonical coordinates we get $c_i(V)=e_i$, which at the same
  time is the $i$-th apex vertex of the tropical standard simplex $\Delta^d$.
  The type of $e_i$ is $\type_{V}(e_i)=(T_1,T_2,\ldots,T_{d+1})$, where some index $l$ is contained in the $j$-th coordinate $T_j$ for $j\neq i$
  if $v_{l,j}=\min\{v_{l,1},v_{l,2},\ldots,v_{l,i}-1,\ldots,v_{l,{d+1}}\}$. This is
  satisfied by all bases $B_l\in\mathcal{B}$ with $j\in B_l$ and $i\notin B_l$. For $j=i$ all 
  indices $l\in[d+1]$ are contained in $T_i$ since the right hand side of
  $v_{l,i}-1=\min\{v_{l,1},v_{l,2},\ldots,v_{l,i}-1,\ldots,v_{l,{d+1}}\}$ is
  smaller or equal than the left hand side in every case.
\end{proof}

Besides the point $\bZero$, the other pseudovertices of a tropical matroid
polytope correspond to unions of its bases.
\begin{lemma}\label{lem:pseudovertices}
  The pseudovertices of $P\in\mathcal{P}_{k,d}$ are  \[\PV(P)=\left\{-e_J\big|J
    = \bigcup_{i\in I}B_i \text{ for some }I\subseteq[n]\right\}.\] 

  The pseudovertices of the first tropical hypersimplex are 
  \[\PV(\Delta^d)=\left\{-e_J\big|J\in\bigcup_{j=1}^{d}\binom{[d+1]}{j}\right\}.\] 

  Let $(T^{(0)}_1,\ldots,T^{(0)}_{d+1})$ be the type of the
  pseudovertex $\bZero$ with respect to $V$ and consider a point $-e_J\in\PV(P)$. 
  If the complement $J^C$ of $J$ is equal to $\{i_1,\ldots,i_r\}$, then the type 
  $(T_1,\ldots,T_{d+1})$ of $-e_J$ with respect to $V$ is given by
  \begin{equation*} \label{eq:typestmp}T_j \ = \
    \begin{cases}
      T^{(0)}_j\setminus(T^{(0)}_{i_1}\cup\cdots\cup T^{(0)}_{i_r})
      & \text{if $j\in J$} \, ,\\
      T^{(0)}_j\cup({T^{(0)}_{i_1}}^C\cap\cdots\cap {T^{(0)}_{i_r}}^C)& \text{otherwise} \, .
    \end{cases}
  \end{equation*}
\end{lemma}
\begin{proof}
  Consider the point $v_J:=c(-e_J)=e_{J^C}$ with canonical coordinates \[v_{J,i}=
  \begin{cases}
    0  & \text{if $i\in J$} \, ,\\
    1& \text{otherwise} \, .
  \end{cases}\] and $\type_{V}(v_J)=(T_1,\ldots,T_{d+1})$. 

  Since the union of the elements of one or more bases of $\mathcal{M}$ consists
  of at least $k$ elements, the index set $J$ has at least $k$ elements and thus
  we have $r\leq d-k+1$ for the cardinality $r$ of $J^C$. 
  We can assume that $J^C=\{1,2,\ldots,r\}$. Then some index $l$ occurs in the 
  $j$-th coordinate $T_j$ if and only if
  \begin{eqnarray}\label{eq:condHypersimplex:typePseudovert}
    v_{l,j}-v_{J,j}&=&\min\{v_{l,1}-1,\ldots,v_{l,r}-1,v_{l,r+1},\ldots,v_{l,d+1}\}\\
    &=&\min\{v_{l,1}-1,\ldots,v_{l,r}-1\}\in \{-1,0\}\nonumber.
  \end{eqnarray}
  For $j\in J$ the left hand side of
  equation~(\ref{eq:condHypersimplex:typePseudovert}) is $v_{l,j}-0\in\{0,1\}$. 
  If $j\in B_l$, we get $v_{l,j}-v_{J,j}=0-0$ and
  this is minimal in~(\ref{eq:condHypersimplex:typePseudovert}) if the coordinates
  $v_{l,i}$ are equal to one for all $i\in J^C$, 
  i.e. $i\notin B_l$. If $j\notin B_l$, we get
  $v_{l,j}-v_{J,j}=1\notin\{-1,0\}$. Therefore, $T_j$ is equal to $\{(l\mid j\in B_l)\,\wedge (i\notin B_l \text{ for all } i\in J^C)\}=T^{(0)}_j\setminus(T^{(0)}_{i_1}\cup\cdots\cup T^{(0)}_{i_r})$.
  
  For $j\in J^C$ the left hand side is $v_{l,j}-1\in\{0,-1\}$. 
  If $j\in B_l$, we get
  $v_{l,j}-v_{J,j}=-1=\min\{v_{l,1}-1,\ldots,v_{l,j}-1,\ldots,v_{l,r}-1\}$. 
  If $j\notin B_l$, we get $v_{l,j}-v_{J,j}=1-1=0$ and this is minimal 
  in~\eqref{eq:condHypersimplex:typePseudovert} if the
  coordinates $v_{l,i}$ are equal to one for all $i\in J^C$, i.e. $i\notin
  B_l$. Therefore, $T_j$ is equal to $\{l\mid j\in B_l\,\text{\bf or } (i\notin
  B_l \text{ for all } i\in J^C)\}=T^{(0)}_j\cup({T^{(0)}_{i_1}}^C\cap\cdots\cap {T^{(0)}_{i_r}}^C)$.

  If $r=d-k+1$, the pseudovertex $v:=c(-e_J)$ is a generator of $P$. Each of
  its type entries contains the index, which is assigned to a basis
  $B\in\mathcal{B}$. Since $B$ is the only basis with
  $i_1,\ldots,i_{d-k+1}\notin B$, its index is the only element of
  $T_j=T^{(0)}_j\setminus(T^{(0)}_{i_1}\cup\cdots\cup T^{(0)}_{i_{d-k+1}})$ for
$j\in B$. For this reason, the generators as defined in~(\ref{eq:generators}) are exactly the tropical vertices of $P$.

  Now we consider the other points of $\PV(V)$, i.e. $r<d-k+1$.
  The intersection of two type entries $T_{j_1}\cap T_{j_2}$ is equal to
  \begin{equation} \label{eq:intersection}T_{j_1}\cap T_{j_2} \ = \
    \begin{cases}
      (T^{(0)}_{j_1}\cap T^{(0)}_{j_2})\setminus(T^{(0)}_{i_1}\cup\cdots\cup T^{(0)}_{i_r})
      & \text{if $j_1,j_2\in J$} \, ,\\
      (T^{(0)}_{j_1}\cap T^{(0)}_{j_2})\cup({T^{(0)}_{i_1}}^C\cap\cdots\cap {T^{(0)}_{i_r}}^C)& \text{otherwise} \, .
    \end{cases}
  \end{equation}
  In the first case of \ref{eq:intersection}, $T_{j_1}\cap T_{j_2}$  
  consists of at least one tropical vertex $v_l$ with $v_{l,j_1}=v_{l,j_2}=0$ and
  $v_{l,i}=1$ for all $i\in J^C$. In the second case there are even
  more tropical vertices allowed and $T_{j_1}\cap T_{j_2}\neq \emptyset$. Hence, Proposition 17 of~\cite{DS04}
  tells us that the cell $X_T$ has dimension $0$, i.e. the given points really
  are pseudovertices of $P$.
  For $J=\bigcup_{i\in I}B_i$ and $J'=\bigcup_{i\in
    I'}B_i$ with $I\neq I'\subseteq[n]$ the tropical line segment	
  between $v_J$ and $v_{J'}$ is the concatenation of the two
  ordinary line segments $[v_J, v_{J\cup \tilde{J}}]$ and
  $[v_{J\cup \tilde{J}},v_{J'}]$. The point $v_{J\cup
    \tilde{J}}$ is again a point of $\PV(P)$. Therefore, there are no
  other pseudovertices as the given points in $\PV(P)$. 
  
  Now we consider the tropical standard simplex $\Delta^d$. If the tropical vertex $v_l:=v_{B_l}$,
  $B_l\in\binom{[d+1]}{1}$, of $\Delta^d$ is given by the vector $v_{B_l}=-e_l$
  ($l=1,\ldots,d+1$), then
  the type of the origin $\bZero$ with respect to $\Delta^d$ is
  $T^{\bZero}=(1,2,\ldots,d+1)$. Therefore, this is an interior point of
  $\Delta^d$. Let $v_J$ with $J\in\bigcup_{j=1}^{d}\binom{[d+1]}{j}$ be any pseudovertex of
$\Delta^d$. Since for $i\in J$ and  $i\notin J$, we have \mbox{$v_{i,i}-v_{J,i}=	
    0=\displaystyle\min\{v_{l,1}-1,\ldots,v_{l,r}-1,v_{l,r+1},\ldots,v_{l,i},\ldots,v_{l,d+1}\}$} and\\\mbox{$v_{i,i}-v_{J,i}=-1=\displaystyle\min\{v_{l,1}-1,\ldots,v_{l,i}-1,\ldots,v_{l,r}-1\}$}, respectively, it follows that the index $i$ is contained in the $i$-th entry of $T$ for all $i=1,\ldots,d+1$, i.e. $T^{\bZero}\subset T$. Hence, $\Delta^d$ is a polytrope.
\end{proof}

Let $v_J= \sum_{i\in J}-e_i=-e_J$ be a pseudovertex of $P$ with 	
$J=\bigcup_{i\in I}B_i$  for $I\subseteq[n]$. If the complement $J^C$ of $J$ is
equal to $\{i_1,i_2,\ldots,i_r\}$ with $r\le d-k+1$, we will denote $v_J$ as	
$e_{i_1,i_2,\ldots,i_r}$ and its type with respect to $P$ as
\[\type_{V}(v_J)=T({v}_J)=\big(T_1({v}_J),\ldots,T_{d+1}({v}_J)\big).\]
Because of the previous lemma, the $i$-th entry of $T({v}_J)$ contains all bases using
edge $i\in J$ that are possible after deleting the edges of $J^C$ in the
corresponding graph $G$ or, equivalently, all bases that are possible after
(re-)inserting edge
$i\in J^C$ into $(V(G),E(G)\setminus\{J^C\})$.

We call a sequence of pseudovertices
$e_{\emptyset},e_{i_1},e_{i_1,i_2},\ldots,e_{i_1,i_2,\ldots,i_{d-k+1}}$, or rather the set \\\mbox{$\{i_1,\ldots,i_{d-k+1}\}\subset[d+1]$}, {\it valid} if 
the edge set $E\setminus\{i_1,\ldots,i_{d-k+1}\}$ contains a spanning tree of the underlying
graph $G$. 
The first point $e_{\emptyset}=\bZero$ is assigned to the total edge set $E$ of
$G$. Then we delete edge after edge such that the graph is
still connected until the edge set forms a connected graph without cycles. So
the last point of a valid sequence is the tropical vertex $v_B$ of $P$
with $B=[d+1]\setminus\{i_1,i_2,\ldots,i_{d-k+1}\}$. 

It turns out that the pseudovertices of the valid
sequences and subsequences of them play a major
role in the calculation of the maximal bounded und unbounded cells of
$P$.

\begin{lemma}
  The maximal bounded cells of $P\in\mathcal{P}_{k,d}$ are of dimension $d-k+1$. They form the tropical convex
  hull of the pseudovertices of a valid sequence
  $\bZero,e_{i_1},e_{i_1,i_2},\ldots,e_{i_1,i_2,\ldots,i_{d-k+1}}$,
  where the last
  pseudovertex is a tropical vertex $v_B$ according to the basis
  $B=[d+1]\setminus\{i_1,i_2,\ldots,i_{d-k+1}\}\in\mathcal{B}$ of $\mathcal{M}$.

  Let $T^{(0)}=(T^{(0)}_1,\ldots,T^{(0)}_{d+1})$ be the type of the
  pseudovertex $\bZero$ with respect to $P$. Then the type
  $T=(T_1,\ldots,T_{d+1})$ of the interior of the bounded cell
  $X_{T}=\tconv(\bZero,e_{i_1},e_{i_1,i_2},\ldots,e_{i_1,i_2,\ldots,i_{d-k+1}})$ is given by 
  $T_{i_1}=T^{(0)}_{i_1},T_{i_2}=T^{(0)}_{i_2}\setminus
  T^{(0)}_{i_1},\ldots,T_{i_{d-k+1}}=T^{(0)}_{i_{d-k+1}}\setminus (T^{(0)}_{i_1}\cup
  T^{(0)}_{i_2}\cup T^{(0)}_{i_{d-k}})$ and $T_j=T^{(0)}_j\setminus (T^{(0)}_{i_1}\cup
  T^{(0)}_{i_2}\cup T^{(0)}_{i_{d-k+1}})$ for all $j\in B$.
\end{lemma}
\begin{proof}
  First, we will show that this sequence really defines a bounded cell
  of $P$, i.e. $T_{j}\neq\emptyset$ for all $j\in [d+1]$.
  So consider the type entry at some coordinate $i_j\in B^C$ 
  \begin{eqnarray*}
    T_{i_j}&=&T_{i_j}(\bZero)\cap T_{i_j}(e_{i_1})\cap \ldots \cap\\ &&
    T_{i_j}(e_{i_1,\ldots,i_{j-1}})\cap\\ &&
    T_{i_j}(e_{i_1,\ldots,i_{j}})\cap\ldots\cap\\ &&
    T_{i_j}(e_{i_1,\ldots,i_{d-k+1}})\\
    &=&\{l\mid i_j\in B_l\}\cap\{l\mid i_j\in B_l \text{ and } i_1\notin
    B_l\}\cap\ldots\cap\\ &&\{l\mid i_j\in B_l \text{ and }
    (i_1,\ldots,i_{j-1}\notin B_l)\}\cap\\ &&
    \{l\mid i_j\in B_l \text{ or } (i_1,\ldots,i_{j}\notin
    B_l)\}\cap\ldots\cap\\ &&
    \{l\mid i_j\in B_l \text{ or } (i_1,\ldots,i_{d-k+1}\notin B_l)\}\\
    &=& \{l\mid i_j\in B_l \text{ and }(i_1,\ldots,i_{j-1}\notin B_l)\}\\
    &=& T^{(0)}_{i_j}\setminus(T^{(0)}_{i_1}\cup\ldots\cup T^{(0)}_{i_{j-1}}).
  \end{eqnarray*}
  The cardinality of $T_{i_j}=T^{(0)}_{i_j}\cap {T^{(0)}_{i_1}}^C\cap\ldots\cap
  {T^{(0)}_{i_{j-1}}}^C$ is equal to the number of tropical vertices
  $v$ of $P$ with $v_{i_j}=0$ and $v_{i_1}=\ldots=v_{i_{j-1}}=1$ (in
  canonical coordinates) respectively to the number of bases $B$ with $i_j\in B$ and
  $i_1,\ldots,i_{j-1}\notin B$, which is greater than $0$ since we consider only
  valid sequences. So every type
  coordinate $T_{i_j}$ contains at least
  one entry. In the case of uniform matroids we have the choice of $d+1-j$ {\it free} coordinates from which
  $k-1$ must be equal to $0$, i.e. the cardinality
  of $T_{i_j}$ is equal to $\binom{d+1-j}{k-1}$.

  Analogously, the other type entries $T_j=T^{(0)}_j\setminus (T^{(0)}_{i_1}\cup
  T^{(0)}_{i_2}\cup T^{(0)}_{i_{d-k+1}})=\{v_B\}$ for $j\in B$ and their cardinality $|T_j|=1$ can be
  verified. Furthermore, we have $T_1\cup \cdots \cup T_{d+1}=[n]$,
  because $T^{(0)}_1\cup \cdots \cup T^{(0)}_{d+1}=[n]$.
  Since no type entry of $T$ is empty, the cell $X_{T}$ is bounded. More precisely,
  $T_{i_1},\ldots,T_{i_{d-k+1}}$ is a partition of the indices of
  $\V(P)\setminus\{v_B\}$, and the other type coordinates each contain the
  index of the tropical vertex $v_B$; we call this a {\it pre-partition}.
  By Proposition 17 in \cite{DS04}, the dimension of $X_{T}$ is $d-k+1$.
  
  Removing one pseudovertex $e_{i_1,\ldots,i_r}$ with $r\in [d-k+1]$ from a valid sequence, we obtain
  $T_{i_{r+1}}=T^{(0)}_{i_{r+1}}\setminus(T^{(0)}_{i_1}\cup\cdots\cup T^{(0)}_{i_{r-1}})$ and
  $T_{i_r}\cap T_{i_{r+1}}\neq\emptyset$. This yields a bounded cell with lower
  dimension than $d-k+1$.

  Adding a pseudovertex $e_J$ to $X_{T}$, $J\neq B$ with $J^C=\{j_1,\ldots,j_r\}$ ($1\leq r\leq d-k+1$) and
  $(j_1,\ldots,j_l)\neq (i_1,\ldots,i_l)$ for all $l=1,\ldots,r$, we
  consider $T'=T\cap\type_{P}(e_J)$. To keep the status of a maximal bounded
  cell, the type of the cell still has to be a pre-partition of $[n]$ without empty
  type entries. There are three different cases (1)-(3).

  (1) For $J^C\not\subseteq B^C$ and $J\cap B\neq\emptyset$, there is an
  index $j\in J\cap B$. We consider the $j$-th type entry of $T'$ that is 
  equal to $T_{j}\cap T^{(0)}_j(e_J)=T^{(0)}_j\cap
  {T^{(0)}_{i_1}}^C\cap\cdots\cap {T^{(0)}_{i_{d-k+1}}}^C\cap {T^{(0)}_{j_1}}^C\cap\cdots\cap
  {T^{(0)}_{j_r}}^C$. This is an empty set since there are no tropical vertices of
  $P$ with $d-k+1+r$ entries equal to one. The cells with empty type entries are
not bounded.

  (2) For $J^C\not\subseteq B^C$ and $J\cap B=\emptyset$, we consider an index
  $j\in J\cap B^C$ that corresponds to a valid sequence with $i_t=j$, $t\in\{1,\ldots,d-k+1\}$. The $j$-th type entry of
  $T'$ is equal to 
  $T^{(0)}_j(e_J)\cap T_j=T^{(0)}_j\cap
  {T^{(0)}_{j_1}}^C\cap\cdots\cap {T^{(0)}_{j_r}}^C\cap
  {T^{(0)}}^C_{i_1}\cap\cdots\cap {T^{(0)}}^C_{i_{t-1}}$. Since
  $J^C\not\subseteq\{i_1,\ldots,i_{t-1}\}$, the cardinality of $T'_j$ is less than $|T^{(0)}_j|$,
  and we get no valid partition of $[n]$.

  (3) For $J^C\subset B^C$ we have $r<d+1-k$ (otherwise $J=B$). We choose the smallest index $j$ such that $i_j\in J\cap B^C$. That means
  $i_1,\ldots,i_{j-1}\in J^C\subset B^C$. Since we have 
  $(i_1,\ldots,i_l)\neq(j_1,\ldots j_l)$ for all 
  $l=1,\ldots,r$, we know that
  $(i_1,\ldots,i_{j-1})\neq(j_1,\ldots,j_r)$ leading to $|T_{i_j}|=|T^{(0)}_{i_j}\cap
  {T^{(0)}_{i_1}}^C\cap\cdots\cap {T^{(0)}_{i_{j-1}}}^C|>|T'_{i_j}|=|T^{(0)}_{i_j}\cap
  {T^{(0)}_{j_1}}^C\cap\cdots\cap {T^{(0)}_{j_r}}^C|$. As in the other two cases this is
  no valid pre-partition of $[n]$.

  In every case the adding of a pseudovertex from another sequence leads to
  unfeasible types of bounded cells. 

  Similarly, it is not difficult to see that removing a pseudovertex and adding 
  a new one from another sequence leads to unfeasible types or lower dimensional
  bounded cells, i.e. mixing of valid sequences is not possible. Altogether, we
get the desired maximal bounded cells of $P$.
\end{proof}

There are $n\cdot(d+1-k)!$ maximal bounded cells of $P$ since we have 
$(d+1-k)!$ possibilities to add edges to a spanning tree until we get the whole graph. 
\begin{example}
  The tropical matroid polytope $P$ from Example \ref{ex:RunEx1} is contained
  in the $4$-dimensional tropical hyperplane with apex $\bZero$. It is shown in
  Figure~\ref{fig:PVG} as the abstract graph of the vertices and edges of its bounded subcomplex. Its maximal bounded cells are ordinary
  simplices of dimension $d-k+1=2$, whose pseudovertices are the tropical
  vertices $V=\{v_{B_1},\ldots,v_{B_8}\}$ (dark), the origin $\bZero$ (the
  centered point) and the five corners $c_i=e_i$ (light).  
  The four tropical vertices with indices $3$,~$4$,~$5$ and $8$ correspond to the
  bases that are possible after deleting edge $1$ in the underlying graph and
  therefore adjacent to the point $e_1$. One valid sequence $i_1,i_2$ leading 
  to a maxim bounded cell is for example the (tropical/ordinary) convex hull of
  $e_{\emptyset}=(0,0,0,0,0),\,e_4=(0,0,0,0,1)$ and $e_{4,2}=v_{B_1}=(0,0,1,0,1)$, 
  i.e. $i_1=4$ and $i_2=2$, with interior cell type $(1,1,36,1,24578)$, 
  representing the basis $B_1=\{1,2,4\}$.
  \begin{figure}[htb]
    \centering
    \psfrag{e0}{$e_0$}
    \psfrag{e1}{$e_1$}
    \psfrag{e2}{$e_2$}
    \psfrag{e3}{$e_3$}
    \psfrag{e4}{$e_4$}
    \psfrag{0}{$1$}
    \psfrag{1}{$2$}
    \psfrag{2}{$3$}
    \psfrag{3}{$4$}
    \psfrag{4}{$5$}
    \psfrag{5}{$6$}
    \psfrag{6}{$7$}
    \psfrag{7}{$8$}
    \subfigure{\includegraphics[trim = 20mm 10mm 20mm 30mm, clip,scale=0.5]{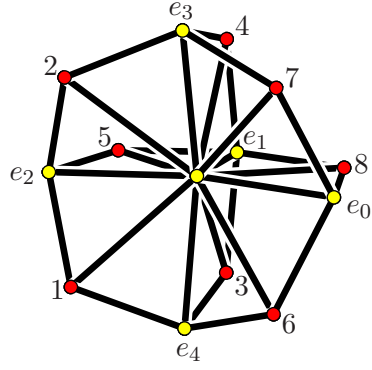}}
    \caption{\label{fig:PVG}The abstract 1-skeleton of the bounded subcomplex of the tropical matroid
polytope of Example~\ref{ex:RunEx1}.}
  \end{figure}
\end{example}

All cells in the tropical complex $\mathcal{C}_{V}$, bounded or not, are pointed, i.e. they do not
contain an affine line. So each cell of $\mathcal{C}_{V}$ must contain
a bounded cell as an ordinary face.

We now state the main theorem about the coarse types of maximal cells in the
cell complex of a tropical matroid polytope. Let $b_{I,J}$ denote the number of bases
  $B\in\mathcal{B}$ with $I\subseteq B$ and $J\subseteq B^C$.

\begin{theorem}\label{thm:coarse types of tmp}
 Let $\mathcal{C}$ be the
tropical complex induced by the tropical vertices of a tropical matroid polytope
$P\in\mathcal{P}_{k,d}$. The set of all coarse types of the maximal cells arising in $\mathcal{C}$ is
given by those tuples $(t_1,\ldots,t_{d+1})$ with
  \begin{equation}\label{eq:coarse types of tmp} t_j \ = \
    \begin{cases}
      b_{\{i_1\},\emptyset}+b_{\,\emptyset,\{i_1,i_2,\ldots,i_{d'+1}\}}  & \text{if $j=i_1$} \, ,\\
      b_{\{i_l\},\{i_1,\ldots,i_{l-1}\}} & \text{if $j=i_l\in \{i_2,\ldots i_{d'+1}\}$} \, ,\\
      0 & \text{otherwise} \, .
    \end{cases}
  \end{equation} 
  where $e_{i_1},\ldots,e_{i_1,i_2,\ldots,i_{d'}}$ form a subsequence of a valid
  sequence of $P$.
\end{theorem}
\begin{proof}
 Depending on the maximal bounded (ordinary) face in the boundary, there are 
  three types of maximal unbounded cells in $\mathcal{C}_{V}$.
  
  The first one, $X_{T}$, contains a maximal bounded cell of dimension $d-k+1$, which is the tropical convex hull of the
  pseudovertices of a complete valid sequence
  $\bZero,e_{i_1},e_{i_1,i_2},\ldots,e_{i_1,i_2,\ldots,i_{d-k+1}}$ where
  $B^C=\{i_1,\ldots,i_{d-k+1}\}$ is the complement of a basis of $\mathcal{M}$. To get full-dimensional we have
  the choice between $k-1$ of $k$ free directions $-e_{i}$, $i\in B$. So let
  $-e^{\infty}_{j_1},\ldots,-e^{\infty}_{j_{k-1}}$ be the {\it extreme rays} of
  $X_{T}$, and $(T^{(0)}_1,\ldots,T^{(0)}_{d+1})$ be the type of the pseudovertex $\bZero$
  with respect to $P$. 
  Then the type
  $T=(T_1,\ldots,T_{d+1})$ of the interior of this unbounded cell
  $X_{T}$ is given by the intersection of the types of its vertices and therefore
  $T_{i_1}=T^{(0)}_{i_1},T_{i_2}=T^{(0)}_{i_2}\setminus
  T^{(0)}_{i_1},\ldots,T_{i_{d-k+1}}=T^{(0)}_{i_{d-k+1}}\setminus (T^{(0)}_{i_1}\cup
  T^{(0)}_{i_2}\cup T^{(0)}_{i_{d-k}})$, $T_i=T^{(0)}_i\setminus (T^{(0)}_{i_1}\cup
  T^{(0)}_{i_2}\cup T^{(0)}_{i_{d-k+1}})$ for $i\notin B^C\cup\{j_1,\ldots,j_{k-1}\}$ and
  $T_{j_1}=\ldots=T_{j_{k-1}}=\emptyset$. Choosing $d'=d-k+1$ and $i_{d'+1}=i$, we
  get the coarse type entries of equation~(\ref{eq:coarse types of tmp}).

  The second type, $X_{T}$, of maximal unbounded cells contains a bounded cell of
  lower dimension $d'\in\{0,\ldots,d-k\}$, which is the tropical convex hull of the
  pseudovertices of some subsequence
  $e_{i_1},e_{i_1,i_2},\ldots,e_{i_1,i_2,\ldots,i_{d'+1}}$.
  To get full-dimensional we still need the extreme rays
  $e_{i_1,i_2,\ldots,i_{d'+1}}-e^{\infty}_l$ for all directions
  $l\notin\{i_1,\ldots,i_{d'+1}\}$. Then the type
  $T=(T_1,\ldots,T_{d+1})$ of the interior of this unbounded cell
  $X_{T}$ is given by 
  $T_{i_1}=T^{(0)}_{i_1}\cup({T^{(0)}_{i_1}}^C\cap\cdots\cap {T^{(0)}_{i_{d'+1}}}^C),T_{i_2}=T^{(0)}_{i_2}\setminus
  T^{(0)}_{i_1},\ldots,T_{i_{d'+1}}=T^{(0)}_{i_{d'+1}}\setminus (T^{(0)}_{i_1}\cup
  T^{(0)}_{i_2}\cup T^{(0)}_{i_{d'}})$, $T_j=\emptyset$ for $j\notin
  \{i_1,\ldots,i_{d'+1}\}$ with the coarse type as given in equation~(\ref{eq:coarse
    types of tmp}).

  The third and last type of maximal unbounded cells contains a bounded cell of
  dimension $d-k$ and is assigned to the non-bases of $\mathcal{M}$, i.e. to
  the subsets of $E$ with cardinality $k$ that are not bases. Let
  $i_1,\ldots,i_{d-k+1}$ be the complement of a non-basis $N$ and 
  $i_1,\ldots,i_{d-k}$ a valid subsequence. Then there is 
  an unbounded cell $X_{T}$ that is the tropical convex hull of the pseudovertices
  $\bZero, e_{i_1},\ldots,e_{i_{d-k}}$ and the extreme rays $\bZero-e^{\infty}_l$ for all directions
  $l\notin\{i_1,\ldots,i_{d-k+1}\}$ and with type entries 
  $T_{i_1}=T^{(0)}_{i_1},T_{i_2}=T^{(0)}_{i_2}\setminus
  T^{(0)}_{i_1},\ldots,T_{i_{d-k+1}}=T^{(0)}_{i_{d-k+1}}\setminus (T^{(0)}_{i_1}\cup
  T^{(0)}_{i_2}\cup T^{(0)}_{i_{d-k}})$, $T_j=\emptyset$ for $j\notin
  \{i_1,\ldots,i_{d-k+1}\}$.  Choosing $d'=d-k$ and observing that
  $b_{\,\emptyset,\{i_1,i_2,\ldots,i_{d'+1}\}}=0$ for the non-basis 
  $\{i_1,i_2,\ldots,i_{d'+1}\}^C$ we get the desired result.
\end{proof}

Restricting ourselves to the uniform case, we get the following result.

\begin{corollary}
 The coarse types of the maximal cells in the tropical complex induced by the
  tropical vertices of the tropical hypersimplex $\Delta_k^d$ in $\TA^d$ with
  $2\le k< d+1$ are up to symmetry of $\Sym(d+1)$ given by
  \[\left(\binom{d+1-\alpha}{k}+\binom{d}{k-1},\,\binom{d-1}{k-1},\ldots,\,\binom{d-(\alpha-1)}{k-1},\underbrace{0,\ldots,0}_{d+1-\alpha}\right)\]
  where $0\le\alpha\le d+2-k$ correlates to the maximal dimension of a bounded
  cell of its boundary. 
\end{corollary}

Now we relate the combinatorial properties of the tropical complex
$\mathcal{C}$ of a tropical
matroid polytope to algebraic properties of a monomial ideal which is
assigned to $\mathcal{C}$. 
As a
direct consequence of Theorem~\ref{thm:DJS2009} and Corollary 3.5
in~\cite{DJS09}, we can state the generators
of the coarse type ideal \[I=\langle x^{\mathbf{t}(p)}\colon
p\in\TA^d\rangle\subset\KK[x_1,\ldots,x_{d+1}],\] where $\mathbf{t}(p)$ is the
coarse type of $p$ and
$x^{\mathbf{t}(p)}={x_1}^{{\mathbf{t}(p)}_1}{x_2}^{{\mathbf{t}(p)}_2}\cdots{x_{d+1}}^{{\mathbf{t}(p)}_{d+1}}$.
\begin{corollary}The coarse type ideal $I$ 
is equal to \[\langle x_{i_1}^{t_{i_1}}x_{i_2}^{t_{i_2}}\cdots x_{i_{d'+1}}^{t_{i_{d'+1}}}\colon
[d+1]\setminus\{i_1,\ldots,i_{d'}\} \text{ contains a basis }\rangle\] with
$(t_{i_1},t_{i_2},\ldots,t_{i_{d'+1}})=\big(b_{\{i_1\},\emptyset}+b_{\,\emptyset,\{i_1,i_2,\ldots,i_{d'+1}\}},b_{\{i_2\},\{i_1\}},\ldots,b_{\{i_{d'+1}\},\{i_1,\ldots,i_{d'}\}}\big)$.
\end{corollary}

\begin{example}The tropical complex $\mathcal{C}$ of the tropical matroid 
polytope of Example~\ref{ex:RunEx1} has 73 maximal cells. There are five maximal cells for the case
$d'=0$ with $t_{i_{d'+1}}=8$ and $t_j=0$ for $j\neq i_{d'+1}$, and 48 for the case 
$d'=2$ according to the 8 bases. Finally, there are 20 maximal cells for the case $d'=1$, where
$[d+1]\setminus\{i_1\}$ contains a basis, but $[d+1]\setminus\{i_1,i_2\}$ does
not necessarily contain a basis.

The coarse type ideal of $\mathcal{C}$ is given by
\begin{eqnarray*}I&=\langle&
{x_1}^1{x_2}^2{x_3}^5,{x_1}^1{x_2}^5{x_3}^2,{x_1}^2{x_2}^1{x_3}^5,{x_1}^4{x_2}^1{x_3}^3,{x_1}^4{x_2}^3{x_3}^1,{x_1}^2{x_2}^5{x_3}^1,{x_2}^2{x_3}^6,{x_2}^6{x_3}^2,\\
&&{x_2}^2{x_3}^5{x_4}^1,{x_2}^5{x_3}^2{x_4}^1,{x_1}^2{x_3}^6,{x_1}^5{x_3}^3,{x_1}^2{x_3}^5{x_4}^1,{x_1}^4{x_3}^3{x_4}^1,{x_3}^8,{x_3}^5{x_4}^3,{x_1}^8,{x_1}^5{x_2}^3,\\
&&{x_1}^5{x_4}^3,{x_1}^4{x_3}^1{x_4}^3,{x_1}^4{x_2}^3{x_4}^1,{x_1}^4{x_2}^1{x_4}^3,{x_0}^2{x_4}^6,{x_4}^8,{x_0}^2{x_1}^1{x_4}^5,{x_0}^1{x_1}^2{x_4}^5,{x_1}^2{x_4}^6,\\
&&{x_0}^1{x_2}^2{x_4}^5,{x_2}^2{x_4}^6,{x_0}^2{x_2}^1{x_4}^5,{x_1}^1{x_2}^2{x_4}^5,{x_1}^2{x_2}^1{x_4}^5,{x_0}^2{x_3}^1{x_4}^5,{x_3}^3{x_4}^5,{x_2}^2{x_3}^1{x_4}^5,\\
&&{x_1}^2{x_3}^1{x_4}^5,{x_0}^1{x_2}^5{x_4}^2,{x_2}^6{x_4}^2,{x_2}^5{x_3}^1{x_4}^2,{x_1}^1{x_2}^5{x_4}^2,{x_0}^2{x_3}^6,{x_0}^1{x_2}^2{x_3}^5,{x_0}^2{x_2}^1{x_3}^5,\\
&&{x_0}^2{x_1}^1{x_3}^5,{x_0}^1{x_1}^2{x_3}^5,{x_0}^2{x_3}^5{x_4}^1,{x_0}^1{x_2}^5{x_3}^2,{x_0}^3{x_2}^5,{x_2}^8,{x_0}^1{x_1}^2{x_2}^5,{x_1}^2{x_2}^6,{x_1}^2{x_2}^5{x_4}^1,\\
&&{x_0}^1{x_1}^4{x_2}^3,{x_0}^6{x_4}^2,{x_0}^5{x_1}^1{x_4}^2,{x_0}^5{x_1}^2{x_4}^1,{x_0}^3{x_1}^4{x_4}^1,{x_0}^1{x_1}^4{x_4}^3,{x_0}^5{x_2}^1{x_4}^2,{x_0}^5{x_2}^3,\\
&&{x_0}^5{x_1}^2{x_2}^1,{x_0}^3{x_1}^4{x_2}^1,{x_0}^5{x_3}^1{x_4}^2,{x_0}^5{x_2}^1{x_3}^2,{x_0}^5{x_3}^2{x_4}^1,{x_0}^6{x_3}^2,{x_0}^5{x_1}^1{x_3}^2,{x_0}^6{x_1}^2,\\
&&{x_0}^3{x_1}^5,{x_0}^5{x_1}^2{x_3}^1,{x_0}^3{x_1}^4{x_3}^1,{x_0}^8,{x_0}^1{x_1}^4{x_3}^3\,\rangle\subseteq
R:=\RR[{x_0},{x_1},{x_2},{x_3},{x_4}]\end{eqnarray*}We obtain its minimal free
resolution, which is induced by $\mathcal{C}$\[\mathcal{F}_{\bullet}^{\mathcal{C}}\colon\,0\rightarrow R^{14}\rightarrow R^{78}\rightarrow
R^{172}\rightarrow R^{180}\rightarrow R^{73}\rightarrow I\rightarrow
0,\]where the exponents $i$ of the free graded $R$-modules $R^i$ correspond to
the entries of the $f$-vector $f(\mathcal{C})=(1,14,78,172,180,73)$ of $\mathcal{C}$.
\end{example}

In (ordinary) convexity swapping between interior and exterior description of a
polytope is a famous problem known as the \emph{convex hull problem}. For a
uniform matroid it is possible to indicate the minimal tropical halfspaces of
its tropical matroid polytope.

\begin{theorem} The tropical hypersimplex $\Delta_k^d$ in $\TA^d$ is the
  intersection of its cornered halfspaces and the tropical halfspaces
  $H(\bZero,I)$, where $I$ is a $(d-k+2)$-element subset of $[d+1]$.
\end{theorem}
\begin{proof}
  For $k=1$ the tropical standard simplex is a polytrope and coincides with its
  cornered hull. For $k\geq 2$ we want to verify the three conditions of Gaubert
and Katz in Proposition 1 of~\cite{GaubertKatz09}.

  Let $T=(T_1,\ldots,T_{d+1})$ be the type of
  the apex $\bZero$ of $H(\bZero,I)$. 
  If a vertex $v\in\V(\Delta_k^d)$ 
  appears in some type entry $T_i$, then the $i$-th (canonical) coordinate of $v$
  is equal to zero. Hence, exactly $k$ entries of $T$ contain the index of
  $v$. Since the cardinality of $I^C=[d+1]\setminus I$ is only $k-1$,
  every tropical vertex of $\Delta_k^d$ is contained in some sector $\overline{S_i}$ with
  $i\in I$, i.e. $\Delta_k^d\subseteq H(\bZero,I)$. 
  
  Consider the complement $I^C$ of $I$. For all $i\in I^C$ there is a tropical vertex $v$ with
  $v_i=0$, i.e. $v\in T_i$. Since the cardinality of $I^C$ is equal to $k-1$ and
  $v$ has $k$ entries equal to zero, there must be an index $j\in I$ such that
  $v_j=0$. We can conclude that $T_i\cap T_j\neq\emptyset$.

  The intersection $T_i\cap T_j$ is not empty for arbitrary $i,j\in[d+1]$, because its cardinality is equal to the number
  of tropical vertices $v$ with $v_i=v_j=0$, which is $\binom{d}{k-1}$ with
  $k>1$. 
  For $i\in I$ and $j\in I^C$, the set $T_i\cap T_j$ consists of all tropical vertices
  $v$ with $v_i=0$ and $v_j=1$ (in canonical coordinates).
  On the other hand, the set $\bigcup_{k\in I\setminus\{i\}}T_k$ contains all
  tropical vertices $v$
  with $v_i=1$. 
  So we get $T_i\cap T_j \not\subset \bigcup_{k\in I\setminus\{i\}}T_k$.

  Hence, we obtain that $H(\bZero,I)$ is a minimal tropical
halfspace, and $\Delta_k^d$ is contained in the intersection of its cornered
hull $\displaystyle \bigcap_{i\in[d+1]}H(e_i,\{i\})$ with $\displaystyle
\bigcap_{I\in\binom{[d+1]}{d-k+2}}H(\bZero,I)$.

We still have to prove that the intersection of the given minimal tropical
halfspaces is contained in $\Delta_k^d$. Let us assume that there is a point
$x\in\TA^d\setminus\Delta_k^d$ with $\type_{\Delta_k^d}(x)_i=\emptyset$. Then for any tropical halfspace $H(\bZero,I)$, $I\in\binom{[d+1]}{d-k+2}$, with $i\in I^C$
 we obtain $x\notin H(\bZero,I)$. 

Consequently, the tropical hypersimplex $\Delta_k^d$ is
the set of all points $x\in\TA^d$ satisfying
\begin{eqnarray*}
\displaystyle\bigoplus_{i\in I}  x_i&\leq&\bigoplus_{j\in I^C} x_j \text{ for
all }I\subseteq [d+1]\text{ with }\lvert I \rvert=d-k+2\\\text{and }
 (-1)\odot x_i&\leq&\bigoplus_{j\neq i} x_j  \text{ for all }i\in[d+1].
\end{eqnarray*}
\end{proof}

\begin{example}The second tropical hypersimplex $\Delta_2^3$ in $\TA^3$ is the intersection
  of the $4$ cornered halfspaces $(c_i,\{i\})$ for $i=1,\ldots,4$ and the tropical
  halfspaces $(\bZero,\{1,2,3\})$, $(\bZero,\{1,2,4\})$, $(\bZero,\{1,3,4\})$ and
  $(\bZero,\{2,3,4\})$ with apex $\bZero\in \TA^d$. 
  The second tropical hypersimplex $\Delta_2^2$ in $\TA^2$ is the intersection
  of the three cornered halfspaces $(c_i,\{i\})$ for $i=1,\ldots,3$ and the tropical
  halfspaces $(\bZero,\{1,2\})$, $(\bZero,\{1,3\})$ and
  $(\bZero,\{2,3\})$ with apex $\bZero\in \TA^2$, see Figure \ref{fig:Delta22}.
  \begin{figure}[htb]
    \centering
    \subfigure[\label{fig:Delta22a}]{\psfrag{v_1}{$(0,-1,-1)$}\psfrag{v_2}{$(-1,0,-1)$}\psfrag{v_3}{$(-1,-1,0)$}\includegraphics[scale=0.65]{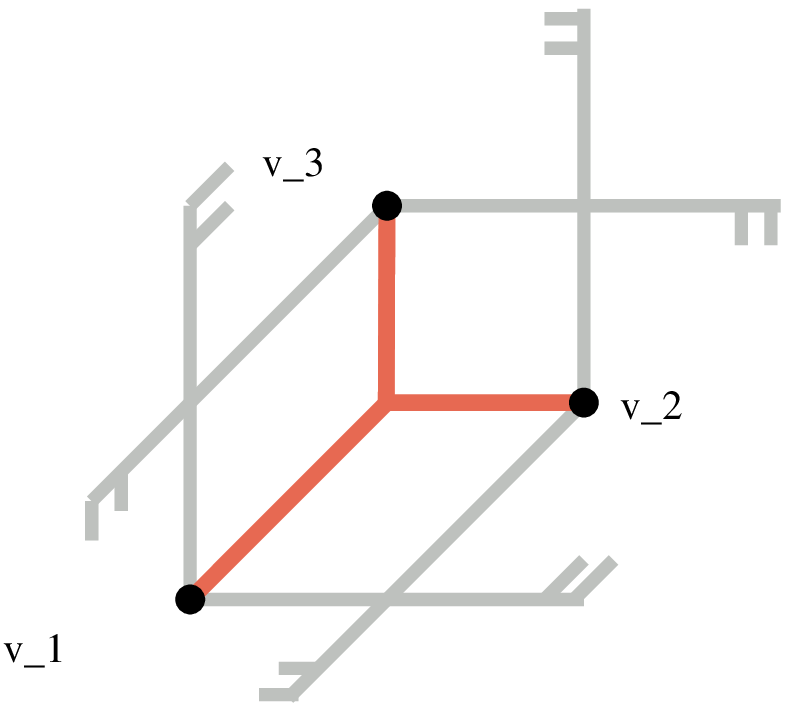}}\hspace*{1cm}
    \subfigure[\label{fig:Delta22b}]{\psfrag{v_1}{$(0,-1,-1)$}\psfrag{v_2}{$(-1,0,-1)$}\psfrag{v_3}{$(-1,-1,0)$}\includegraphics[scale=0.65]{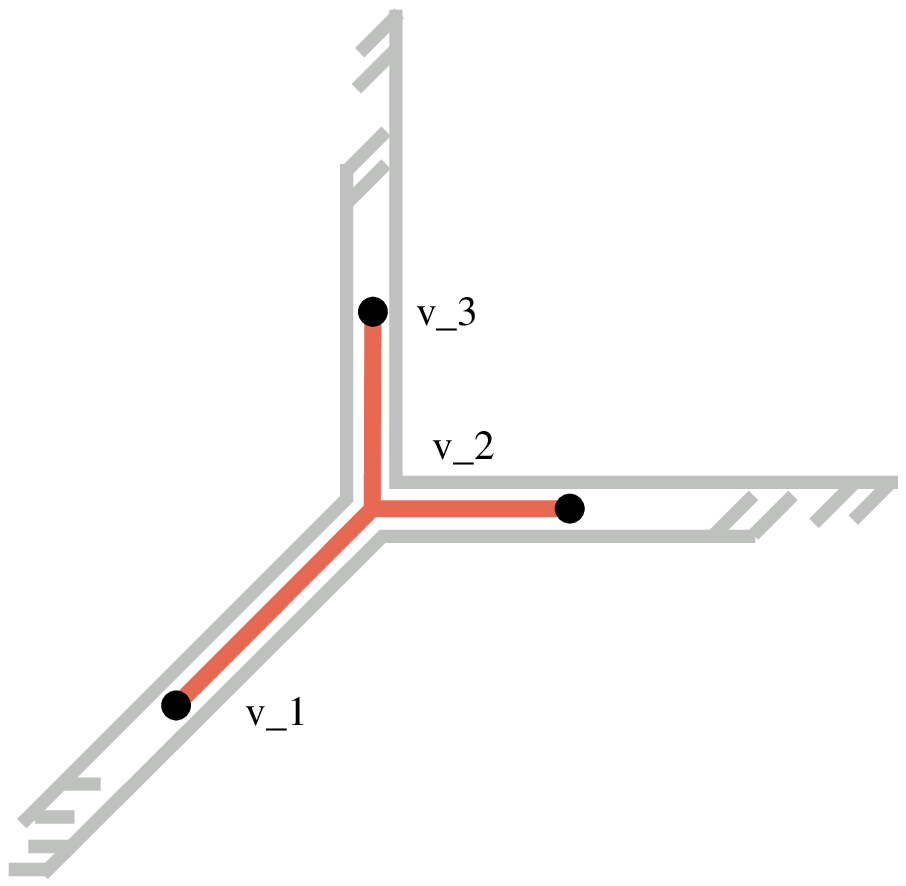}}
    \caption{\label{fig:Delta22}
The tropical hypersimplex $\Delta_2^2$ (dark) is given as the intersection of its
cornered halfspaces (light coloured in Figure~\ref{fig:Delta22a}) and the minimal tropical halfspaces
$(\bZero,\{1,2\})$, $(\bZero,\{1,3\})$  
(light coloured in Figure~\ref{fig:Delta22b}).}
  \end{figure}
\end{example}

\noindent\emph{Acknowledgements.} I would like to thank my advisor Michael
Joswig for suggesting the problem, and for
supporting me writing this article.

\bibliographystyle{amsplain}\def\cprime{$'$}
\providecommand{\bysame}{\leavevmode\hbox to3em{\hrulefill}\thinspace}
\providecommand{\MR}{\relax\ifhmode\unskip\space\fi MR }
\providecommand{\MRhref}[2]{
  \href{http://www.ams.org/mathscinet-getitem?mr=#1}{#2}
}
\providecommand{\href}[2]{#2}

\end{document}